\documentclass[a4paper,12pt]{article}

\usepackage{amsmath,amsfonts,amsthm}
\usepackage{mathrsfs}
\usepackage{bbold}
\usepackage{graphicx}

\newcommand{\1}{{\mathbb{1}}}
\newcommand{\Liminf}{\mathop{\underline{\lim}}\limits}
\newcommand{\argmax}{\mathop{\rm argmax}\limits}
\newcommand{\Ex}{\mathop{\mathbf{\kern 0pt E}}\nolimits}
\newcommand{\Pb}{\mathop{\mathbf{\kern 0pt P}}\nolimits}
\newcommand{\PP}{\mathop{\mathbb{\kern 0pt P}}\nolimits}
\newcommand{\EE}{\mathop{\mathbb{\kern 0pt E}}\nolimits}

\newtheorem{theorem}{Theorem}
\newtheorem{lemma}{Lemma}
\newtheorem{proposition}{Proposition}
\newtheorem{definition}{Definition}

\begin{document}

\title{On Hypothesis Testing for Poisson Processes. Regular Case}

\author{S. \textsc{Dachian}\\
{\small Universit\'e Blaise Pascal, Clermont-Ferrand, France}\\[8pt]
Yu. A. \textsc{Kutoyants}\\
{\small Universit\'e du Maine, Le Mans, France and}\\
{\small Higher School of Economics, Moscow, Russia}\\[8pt]
L. \textsc{Yang}\footnote{Corresponding author. E-mail address:
  Lin\_Yang.Etu@univ-lemans.fr}\\
{\small Universit\'e du Maine, Le Mans, France}}

\date{}

\maketitle

\begin{abstract}
We consider the problem of hypothesis testing in the situation when the first
hypothesis is simple and the second one is local one-sided composite.  We
describe the choice of the thresholds and the power functions of the Score
Function test, of the General Likelihood Ratio test, of the Wald test and of
two Bayes tests in the situation when the intensity function of the observed
inhomogeneous Poisson process is smooth with respect to the parameter. It is
shown that almost all these tests are asymptotically uniformly most powerful.
The results of numerical simulations are presented.
\end{abstract}


\medskip


\section{Introduction}

The hypothesis testing theory is a well developed branch of mathematical
statistics~\cite{LR05}. The asymptotic approach allows to find satisfactory
solutions in many different situations. The simplest problems, like the
testing of two simple hypotheses, have well known solutions. Recall that if we
fix the first type error and seek the test which maximizes the power, then we
obtain immediately (by Neyman-Pearson lemma) the most powerful test based on
the likelihood ratio statistic. The case of composite alternative is more
difficult to treat and here the asymptotic solution is available in the
regular case. It is possible, using, for example, the Score Function test
(SFT), to construct the asymptotically (locally) most powerful test. Moreover,
the General Likelihood Ratio Test (GLRT) and the Wald test (WT) based on the
maximum likelihood estimator are asymptotically most powerful in the same
sense. In the non regular cases the situation became much more complex. First
of all, there are different non regular (singular) situations. Moreover, in
all these situations, the choice of the asymptotically best test is always an
open question.

This work is an attempt to study all these situations on the model of
inhomogeneous Poisson processes. This model is sufficiently simple to allow us
to realize the construction of the well known tests (SFT, GLRT, WT) and to
verify that these test are asymptotically most powerful also for this model,
in the case when it is regular. In the next paper we study the behavior of
these tests in the case when the model is singular. The ``evolution of the
singularity'' of the intensity function is the following: regular case (finite
Fisher information, this paper), continuous but not differentiable (cusp-type
singularity, \cite{DKY-2}), discontinuous (jump-type singularity,
\cite{DKY-2}).  In all the three cases we describe the tests
analytically. More precisely, we describe the test statistics, the choice of
the thresholds and the behavior of the power functions for local alternatives.

Note that the notion of \textit{local alternatives} is different following the
type of regularity/singularity. Suppose we want to test the simple hypothesis
$\vartheta =\vartheta _1$ against the one-sided alternative $\vartheta
>\vartheta _1$. In the regular case, the local alternatives are usually given
by \smash{$\vartheta =\vartheta _1+\frac{u}{\sqrt{n}}$}, $u>0$. In the case of
a cusp-type singularity, the local alternatives are introduced by $\vartheta
=\vartheta _1+u\,n^{-\frac{1}{2\kappa +1}}$, $u>0$. As to the case of a
jump-type singularity, the local alternatives are $\vartheta =\vartheta
_1+\frac{u}{n}$, $u>0$. In all these problems, the most interesting for us
question is the comparison of the power functions of different tests. In
singular cases, the comparison is done with the help of numerical
simulations. The main results concern the limit likelihood ratios in the
non-regular situations.  Let us note, that in many other models of
observations (i.i.d., time series, diffusion processes etc.) the likelihood
ratios have the same limits as here (see, for example, \cite{IH81} and
\cite{DK03}). Therefore, the results presented here are of more universal
nature and are valid for any other (non necessarily Poissonian) model having
one of considered here limit likelihood ratios.

We recall that $X=\left(X_t,\ t\geq 0 \right)$ is an inhomogeneous Poisson
process with intensity function $\lambda \left(t\right)$, $t\geq 0$, if
$X_0=0$ and the increments of $X$ on disjoint intervals are independent and
distributed according to the Poisson law
$$
\Pb\left\{X_t-X_s=k\right\}=\frac{\left(\int_s^t \lambda
\left(t\right) {\rm d}t\right)^k }{k!}\exp\left\{-\int_s^t \lambda
\left(t\right) {\rm d}t\right\}.
$$
In all statistical problems considered in this work, the intensity functions
are periodic with some known period $\tau$ and depend on some one-dimensional
parameter, that is, $\lambda \left(t\right)=\lambda \left(\vartheta
,t\right)$. The basic hypothesis and the alternative are always the same:
$\vartheta =\vartheta _1$ and $\vartheta >\vartheta _1$. The diversity of
statements corresponds to different types of regularity/singularity of the
function $\lambda \left(\vartheta ,t\right)$. The case of unknown period
$\tau$ needs a special study.

The hypothesis testing problems (or closely related properties of the
likelihood ratio) for inhomogeneous Poisson processes were studied by many
authors (see, for example, Brown \cite{Bro71}, Kutoyants \cite{Kut77}, L\'eger
and Wolfson \cite{LW87}, Liese and Lorz \cite{LL97}, Sung \textit{et
  al.}~\cite{STS}, Fazli and Kutoyants \cite{FK05}, Dachian and Kutoyants
\cite{DK09} and the references therein). Note finally, that the results of
this study will appear later in the work \cite{Kut13}.

\section{Auxiliary results}

For simplicity of exposition we consider the model of $n$ independent
observations of an inhomogeneous Poisson process:
$X^n=\left(X_1,\ldots,X_n\right)$, where $X_j=\left(X_j\left(t\right),\ 0\leq
t\leq \tau \right)$, $j=1,\ldots,n$, are Poisson processes with intensity
function $\lambda \left(\vartheta ,t\right)$, $0\leq t\leq \tau $. Here
$\vartheta\in\Theta = [\vartheta _1,b)$, $b<\infty $, is a one-dimensional
  parameter. We have
$$
\Ex_\vartheta X_j\left(t\right)=\Lambda \left(\vartheta
,t\right)=\int_{0}^{t}\lambda \left(\vartheta ,s\right)\;{\rm d}s
$$
where $\Ex_\vartheta $ is the mathematical expectation in the case when the
true value is $\vartheta$. Note that this model is equivalent to the one,
where we observe an inhomogeneous Poisson process $X^T=\left(X_t,\ 0\leq t\leq
T\right)$ with periodic intensity $\lambda(\vartheta,t+j\tau
)=\lambda(\vartheta,t )$, $j=1,2,\ldots,n-1$, and $T=n\tau $ (the period $\tau
$ is supposed to be known). Indeed, if we put $X_j\left(s\right)=X_{s+\tau
  \left(j-1\right)}-X_{\tau \left(j-1\right)}$, $s\in\left[0,\tau \right]$,
$j=1,\ldots,n$, then the observation of one trajectory $X^T$ is equivalent to
$n$ independent observations $X_1,\ldots,X_n$.

The intensity function is supposed to be separated from zero on $\left[0,\tau
  \right]$. The measures corresponding to Poisson processes with different
values of $\vartheta $ are equivalent. The likelihood function is defined by
the equality (see Liese~\cite{Lie75})
\begin{align*}
 L(\vartheta
,X^n)=\exp\left\{\sum_{j=1}^{n}\int_{0}^{\tau }\ln\lambda
\left(\vartheta ,t\right){\rm d}X_j\left(t\right)-n\int_{0}^{\tau
}\left[\lambda
  \left(\vartheta ,t\right)-1 \right]{\rm d}t\right\}
\end{align*}
and the likelihood ratio function is
$$
 L\left(\vartheta ,\vartheta
_1,X^n\right)=L\left(\vartheta,X^n\right)/L\left(\vartheta_1,X^n\right).
$$

We have to test the following two hypotheses
\begin{align*}
\mathscr{ H}_1\quad &:\qquad \vartheta =\vartheta _1,\\
\mathscr{ H}_2\quad &:\qquad \vartheta >\vartheta _1.
\end{align*}
A test $\bar\psi _n=\bar\psi _n\left(X^n\right)$ is defined as the probability
to accept the hypothesis~$\mathscr{ H}_2$. Its power function is
$\beta\left(\bar \psi_n,\vartheta \right)=\Ex_{\vartheta } \bar \psi_n(X^n)$,
$\vartheta >\vartheta _1$.

Denote $\mathcal{ K}_\varepsilon$ the class of tests $\bar \psi _n$ of
asymptotic size $\varepsilon\in\left[ 0,1\right]$:
$$
\mathcal{ K}_\varepsilon =\left\{\bar \psi _n\quad :\quad
\lim_{n\rightarrow
  \infty }\Ex_{\vartheta _1}\bar\psi _n\left(X^n\right)=\varepsilon \right\}.
$$
Our goal is to construct tests which belong to this class and have some
proprieties of asymptotic optimality.

The comparison of tests can be done by comparison of their power functions. It
is known that for any reasonable test and for any fixed alternative the power
function tends to ~$1$. To avoid this difficulty, as usual, we consider
\textit{close} or \textit{contiguous} alternatives. We put $\vartheta
=\vartheta _1+ \varphi _nu$, where $u\in \mathbb{U}_n^+=\bigl[0,\varphi
  _n^{-1}\left(b-\vartheta _1\right)\bigr)$, $\varphi _n=\varphi
  _n\left(\vartheta _1\right)>0$ and $\varphi _n\rightarrow 0$. The rate of
  convergence $\varphi _n\rightarrow 0$ must be chosen so that the normalized
  likelihood ratio
$$
Z_n\left(u\right)=\frac{L\left(\vartheta _1+\varphi
  _nu,X^n\right)}{L\left(\vartheta _1,X^n\right)} ,\qquad u\geq 0,
$$
has a non degenerate limit. In the regular case this rate is usually $\varphi
_n=n^{-1/2}$.

Then the initial problem of hypotheses testing can be rewritten as
\begin{align*}
\mathscr{ H}_1\quad &:\qquad u =0,\\
\mathscr{ H}_2\quad &:\qquad u>0.
\end{align*}

The power function of a test $\bar\psi _n$ is now denoted as
$$
\beta\left(\bar\psi _n,u\right)=\Ex_{\vartheta_1+\varphi _n u}\,\bar\psi
_n,\qquad u>0.
$$

The asymptotic optimality of tests is introduced with the help of the
following definition (see \cite{GR}).

\begin{definition}
We call a test $\psi _n^\star\left(X^n\right) \in \mathcal{ K}_\varepsilon$
locally asymptotically uniformly most powerful (LAUMP) in the class $\mathcal{
  K}_\varepsilon $ if its power function $\beta\left(\psi _n^\star,u\right)$
satisfies the relation: for any test $\bar\psi _n\left(X^n\right)\in \mathcal{
  K}_\varepsilon$ and any $K>0$ we have
$$
\Liminf_{n\rightarrow \infty }\ \inf_{0<u\leq K}\left[\beta\left(\psi
  _n^\star,u\right)-\beta\left(\bar \psi _n,u\right) \right]\geq 0.
$$
\end{definition}

Below we show that in the regular case many tests are LAUMP. In the next paper
\cite{DKY-2}, where we consider some singular situations, a ``reasonable''
definition of asymptotic optimality of tests is still an open question. That
is why we use numerical simulations to compare the tests in \cite{DKY-2}.

We assume that the following \textit{Regularity conditions} are satisfied.

\bigskip

\noindent\textbf{Smoothness.}
 \textit{The intensity function $\lambda \left(\vartheta ,t\right)$, $0\leq
   t\leq \tau $, of the observed Poisson process $X^n$ is two times
   continuously differentiable w.r.t.\ $\vartheta $, is separated from zero
   uniformly on $\vartheta \geq \vartheta _1$, and the Fisher information is
   positive:
$$
{\rm I}\left(\vartheta \right)=\int_{0}^{\tau }\frac{\dot\lambda
  \left(\vartheta ,t\right)^2}{\lambda \left(\vartheta ,t\right)}{\rm
  d}t,\qquad \inf_{\vartheta \in\Theta }{\rm I}\left(\vartheta \right) >0. 
$$
Here $\dot\lambda$ denotes the derivative of $\lambda$ w.r.t.\ $\vartheta $
and, at the point $\vartheta _1$, the derivative is from the right.}

\bigskip

\noindent\textbf{Distinguishability.}
\textit{For any $\nu >0$, we have
$$
\inf_{\vartheta_* \in \Theta }\inf_{\left|\vartheta -\vartheta _*\right|>\nu
}\left\|\sqrt{\lambda \left(\vartheta ,\cdot \right)}-\sqrt{\lambda
  \left(\vartheta_1 ,\cdot \right)} \right\|_{L^2}>0.
$$
Here
$$
\left\|h\left(\cdot \right)\right\|_{L^2}^2=\int_{0}^{\tau
}h\left(t\right)^2{\rm d}t.
$$}

\bigskip

In this case, the natural normalization function is $\varphi _n=n^{-1/2}$ and
the change of variables is $\vartheta =\vartheta _1+\frac{u}{\sqrt{n}}$.

The key propriety of statistical problems in the regular case is the
\textit{local asymptotic normality} (LAN) of the family of measures of
corresponding inhomogeneous Poisson processes at the point $\vartheta_1$.
This means that the normalized likelihood ratio
$$
\tilde Z_n\left(u\right)=L\left(\vartheta _1+ \frac{u}{\sqrt{n}},
\vartheta
 _1,X^n\right)
$$
admits the representation
$$
\tilde Z_n\left(u\right)=\exp\left\{u\,\widetilde{\Delta}
_n\left(\vartheta _1,X^n\right)-\frac{u^2}{2}\,{\rm
I}\left(\vartheta _1\right)+r_n \right\},
$$
where (using the central limit theorem) we have
\begin{align*}
\widetilde{\Delta} _n\left(\vartheta
_1,X^n\right)&=\frac{1}{\sqrt{n}}\sum_{j=1}^{n}\int_{0}^{\tau
}\frac{\dot\lambda \left(\vartheta _1,t\right)}{\lambda \left(\vartheta
  _1,t\right)}\left[{\rm d}X_j\left(t\right)-\lambda \left(\vartheta
  _1,t\right){\rm d}t \right]\\
&\Longrightarrow \widetilde{\Delta} \sim \mathcal{ N}\left(0,{\rm
  I}\left(\vartheta _1\right) \right)
\end{align*}
(convergence in distribution under $\vartheta _1$), and $
r_n=r_n\left(\vartheta _1,u,X^n\right)\overset{p}{\longrightarrow} 0$
(convergence in probability under $\vartheta _1$).  Moreover, the last
convergence is uniform on $0\leq u<K$ for any $K>0$.

Let us now briefly recall how this representation was obtained in
\cite{Kut77}. Denoting $\lambda_0=\lambda \left(\vartheta _1,t\right)$ and
$\lambda _u=\lambda \left(\vartheta _1+\frac{u}{\sqrt{n}},t\right)$, with the
help of the Taylor series expansion we can write
\begin{align*}
\ln Z_n\left(u\right)&=\sum_{j=1}^{n}\int_{0}^{\tau }\ln\frac{\lambda
  _u}{\lambda _0}\left[{\rm d}X_j\left(t\right)-\lambda _0{\rm d}t\right]-
n\int_{0}^{\tau }\left[\lambda _u-\lambda _0-\lambda _0\ln \frac{\lambda
    _u}{\lambda _0}\right] {\rm
  d}t\\ &=\frac{u}{\sqrt{n}}\sum_{j=1}^{n}\int_{0}^{\tau }\frac{\dot \lambda
  _0}{\lambda _0}\left[{\rm d}X_j\left(t\right)-\lambda _0{\rm
    d}t\right]-\frac{u^2}{2}\int_{0}^{\tau }\frac{\dot \lambda _0^2}{\lambda
  _0}{\rm d}t+r_n\\ &=u\widetilde{\Delta} _n\left(\vartheta
_1,X^n\right)-\frac{u^2}{2}{\rm I}\left(\vartheta _1\right)+r_n\Longrightarrow
\widetilde{\Delta} -\frac{u^2}{2}{\rm I}\left(\vartheta _1\right).
\end{align*}

In the sequel, we choose reparametrizations which lead to \textit{universal}
in some sense limits. For example, in the regular case, we put
$$
\varphi _n=\varphi _n\left(\vartheta _1\right)=\frac{1}{\sqrt{n{\rm
      I}\left(\vartheta _1\right)}}\,,\qquad
u\in\mathbb{U}_n^+=\bigl[0,\varphi _n^{-1}\left(b-\vartheta _1\right)\bigr).
$$
With such change of variables, we have
$$
Z_n\left(u\right)=L\left(\vartheta _1+ u\varphi_n, \vartheta
_1,X^n\right)=\exp\left\{u\,\Delta _n\left(\vartheta
_1,X^n\right)-\frac{u^2}{2}+r_n \right\},
$$
where
$$
\Delta _n\left(\vartheta _1,X^n\right)=\frac1{\sqrt{{\rm I}\left(\vartheta
    _1\right)}}\,\widetilde{\Delta }_n \Longrightarrow \Delta \sim \mathcal{
  N}\left(0,1\right).
$$

The LAN families have many remarkable properties and some of them will be used
below.

Let us remind here one general result which is valid in a more general
situation. We suppose only that the normalized likelihood ratio
$Z_n\left(u\right)$ converges to some limit $Z\left(u\right)$ in
distribution. Note that this is the case in all our regular and singular
problems. The following property allows us to calculate the distribution under
local alternative when we know the distribution under the null hypothesis.
Moreover, it gives an efficient algorithm for calculating power functions in
numerical simulations.

\begin{lemma}[Le Cam's Third Lemma]
Suppose that $(Z_n\left(u\right),$ $Y_n)$ converges in distribution under
$\vartheta _1$:
$$
\left(Z_n\left(u\right),Y_n\right)\Longrightarrow
\left(Z\left(u\right),Y\right).
$$
Then, for any bounded continuous function $g\left(\cdot \right)$, we have
\begin{equation*}
\Ex_{\vartheta _1+\varphi _n u} \left[g\left(Y_n\right)\right]\longrightarrow
\Ex \left[Z\left(u\right) g\left(Y\right)\right].
\end{equation*}
\end{lemma}

For the proof see \cite{LC-Y}.

In the regular case, the limit of $Z_n\left(\cdot\right)$ is the random
function
$$
Z\left(u\right)=\exp \left\{u\,\Delta -\frac{u^2}{2}\right\},\qquad
u\geq 0.
$$
So, for any fixed $u>0$, we have the convergence
$$
Z_n\left(u\right)\Longrightarrow Z\left(u\right).
$$
According to this lemma, we can write the following relations for the
characteristic function of $\Delta _n=\Delta _n\left(\vartheta _1,X^n\right)$:
\begin{align*}
\Ex_{\vartheta _1+\varphi_n u}e^{i\mu \Delta _n}\longrightarrow \Ex
Z\left(u\right)e^{i\mu\Delta }=e^{-\frac{u^2}{2}}\Ex e^{u\Delta +i\mu \Delta
}=e^{i\mu u -\frac{\mu ^2}{2}} =\Ex e^{i\mu \left(u + \Delta \right) },
\end{align*}
which yields the asymptotic distribution of the statistic $\Delta _n$ under
the alternative $\vartheta _1+\varphi_n u$:
$$
\Delta _n\left(\vartheta _1,X^n\right)\Longrightarrow u + \Delta\;\sim\; \mathcal{
  N}\left(u,1 \right).
$$

\section{Weak convergence}

All the tests considered in this paper are functionals of the normalized
likelihood ratio $Z_n\left(\cdot \right) $. For each of them, we have to
evaluate two quantities. The first one is the threshold, which guarantees the
desired asymptotic size of the test, and the second one is the limit power
function, which has to be calculated under alternative. Our study is based on
the weak convergence of the likelihood ratio $Z_n\left(\cdot \right)$ under
hypothesis (to calculate the threshold) and under alternative (to calculate
the limit power function). Note that the test statistics of all the tests are
continuous functionals of $Z_n\left(\cdot \right)$.  That is why the weak
convergence of $Z_n\left(\cdot \right)$ allows us to obtain the limit
distributions of these statistics.

We denote $\Pb_\vartheta$ the distribution that the observed inhomogeneous
Poisson processes $X^n$ induce on the measurable space of their
realizations. The measures in the family $\left\{\Pb_\vartheta,\ \vartheta
\geq \vartheta _1\right\}$ are equivalent, and the normalized likelihood ratio
is
\begin{align*}
\ln Z_n\left(u\right)&=\sum_{j=1}^{n}\int_{0}^{\tau }\ln\frac{\lambda
  \left(\vartheta_1 +\varphi _n\left(\vartheta_1\right)u,t\right)}{\lambda
  \left(\vartheta_1 ,y\right)}{\rm d}X_j\left(t\right)\\
&\quad -n\int_{0}^{\tau }\left[{\lambda \left(\vartheta_1 +\varphi
    _n\left(\vartheta_1 \right)u,t\right)}-{\lambda \left(\vartheta_1
    ,t\right)}\right]{\rm d}t ,
\end{align*} 
where $u\in\mathbb{U}_n^+=\bigl[0, \varphi _n^{-1} \left(b-\vartheta_1 \right)
  \bigr)$.  We define $Z_n\left(u\right)$ to be linearly decreasing to zero on
  the interval $\bigl[ \varphi _n^{-1} \left(b-\vartheta_1 \right), \varphi
    _n^{-1} \left(b-\vartheta_1 \right)+1 \bigr]$ and we put
  $Z_n\left(u\right)=0$ for $u> \varphi _n^{-1} \left(b-\vartheta_1
  \right)+1$. Now the random function $Z_n\left(u\right)$ is defined
  on~$\mathbb{R}_+$ and belongs to the space
  $\mathscr{C}_0\left(\mathbb{R}_+\right)$ of continuous on $\mathbb{R}_+$
  functions such that $z\left(u\right)\rightarrow 0$ as $u\rightarrow
  \infty$. Introduce the uniform metric in this space and denote~$\mathcal{B}$
  the corresponding Borel sigma-algebra. The next theorem describes the weak
  convergence under the alternative $\vartheta =\vartheta _1+\varphi _nu_*$
  (with fixed $u_*>0$) of the stochastic process
  $(Z_n\left(u\right),\ u\in\mathbb{R}_+)$ to the process
$$
Z\left(u,u_*\right)=\exp\left\{u\Delta+u u_* -\frac{u^2}{2}\right\},\quad u\in
\mathbb{R},
$$
in the measurable space $\left(\mathscr{ C}_0\left(\mathbb{R}_+\right),
\mathcal{B}\right)$.  Note that in \cite{Kut98} this theorem was proved for a
fixed true value $\vartheta $.  In the hypothesis testing problems considered
here, we need this convergence both under hypothesis $\mathscr{H}_1$, that is,
for fixed true value $\vartheta =\vartheta _1$ ($u_*=0$), and under
alternative $\mathscr{H}_2$ with ``moving'' true value
$\vartheta=\vartheta_{u_*}=\vartheta_1+\varphi _nu_*$.

\begin{theorem}
\label{T1}
Let us suppose that the Regularity conditions are fulfilled.  Then, under
alternative $\vartheta_{u_*}$, we have the weak convergence of the stochastic
process $Z_n=\left(Z_n\left(u\right),\ u\geq 0\right)$ to
$Z=\left(Z\left(u,u_*\right),\ u\geq 0\right)$.
\end{theorem}

According to ~\cite[Theorem 1.10.1]{IH81}, to prove this theorem it is
sufficient to verify the following three properties of the process
$Z_n\left(\cdot \right)$.

\begin{enumerate}
\item The finite-dimensional distributions of $Z_n\left(\cdot \right)$
  converge, under alternative $\vartheta_{u_*}$, to the finite-dimensional
  distributions of $Z\left(\cdot ,u_*\right)$.
\item The inequality
$$
\Ex_{\vartheta_{u_*}}\left|Z_n^{1/2}(u_2)-Z_n^{1/2}(u_1)\right|^2\leq
C\left|u_2-u_1\right|^2
$$
holds for every $u_1,u_2\in\mathbb{U}_n^+$ and some constant $C>0$.
\item There exists $d>0$, such that for some $n_0>0$ and all $n\geq n_0$ we
  have the estimate
$$
\Pb_{\vartheta_{u_*}}\left\{Z_n(u)>e^{-d\left|u-u_*\right|^2}\right\}\leq
e^{-d\left|u-u_*\right|^2}.
$$
\end{enumerate}

Let us rewrite the random function $Z_n\left(\cdot\right)$ as follows:
\begin{align*}
Z_n\left(u\right)&=L\left({\vartheta_1+u\varphi_n},\vartheta_1,X^n\right)\\
& ={L\left({\vartheta_1+u_*\varphi_n},\vartheta_1,X^n\right)}\;
{L\left({\vartheta_1+u\varphi_n},\vartheta_1+u_*\varphi_n,X^n\right)}.
\end{align*}
For the first term we have
$$
L\left({\vartheta_1+u_*\varphi_n},\vartheta_1,X^n\right)
=L\left({\vartheta_{u_*}-u_*\varphi_n},\vartheta_{u_*},X^n\right)^{-1}\Longrightarrow
\exp\left\{u_*\Delta +\frac{u_*^2}{2}\right\}.
$$
Therefore we only need to check the conditions 2--3 for the term
$$
Z_n\left(u,u_*\right)=
L\left({\vartheta_1+u\varphi_n},\vartheta_1+u_*\varphi_n,X^n\right)=
L\left(\vartheta_{u_*}+(u-u_*)\varphi_n,\vartheta_{u_*},X^n\right).
$$

\begin{lemma}
The finite-dimensional distributions of $Z_n\left(\cdot \right)$ converge,
under alternative $\vartheta_{u_*}$, to the finite-dimensional distributions
of $Z\left(\cdot ,u_*\right)$.
\end{lemma}

\begin{proof}
The limit process for $Z_n\left(\cdot,u_*\right) $ is
$$
\exp\left\{\left(u-u_*\right)\Delta
-\frac{\left(u-u_*\right)^2}{2}\right\},\quad u\in\mathbb{R}_+.
$$
Hence
$$
Z_n\left(u\right)\Longrightarrow \exp\left\{u_*\Delta +\frac{u_*^2}{2}\right\}
\exp\left\{\left(u-u_*\right)\Delta
-\frac{\left(u-u_*\right)^2}{2}\right\}=Z\left(u,u_*\right).
$$
For the details see, for example, \cite{Kut98}.
\end{proof}

\begin{lemma}
Let the Regularity conditions be fulfilled. Then there exists a constant
$C>0$, such that
$$
\Ex_{\vartheta_{u_*}}\bigl|{Z}_n^{1/2}(u_1,u_*)-{Z}_n^{1/2}(u_2,u_*)\bigr|^2\leq
C\left|u_1-u_2\right|^2
$$
for all $u_*,u_1,u_2\in\mathbb{U}_n^+$ and sufficiently large values of $n$.
\end{lemma}

\begin{proof}
According to~\cite[Lemma 1.1.5]{Kut98}, we have:
\begin{align*}
\Ex_{\vartheta_{u_*}}&\bigl|{Z}_n^{1/2}(u_1,u_*)-{Z}_n^{1/2}(u_2,u_*)\bigr|^2\\
&\leq n\int_{0}^{\tau}
\left(\frac{\lambda^{1/2}(\vartheta_1+u_1\varphi_n,t)}{\lambda^{1/2}(\vartheta_1+u_*\varphi_n,t)}
-\frac{\lambda^{1/2}(\vartheta_1+u_2\varphi_n,t)}{\lambda^{1/2}(\vartheta_1+u_*\varphi_n,t)}\,
\right) ^2\,\lambda (\vartheta_1+u_*\varphi_n,t)\,\mathrm{d} t\\
&=n\int_{0}^{\tau}\bigl(\lambda^{1/2}(\vartheta_1+u_1\varphi_n,t)
-\lambda^{1/2}(\vartheta_1+u_2\varphi_n,t)\,\bigr)^2\,\mathrm{d} t\\
&= \frac{n}{4}\,\varphi_n^2\left(u_2-u_1\right)^2\int_{0}^{\tau}
\frac{\dot\lambda\left(\vartheta_v,t\right)^2}
     {\lambda\left(\vartheta_v,t\right)}\,\mathrm{d} t\leq C\left(u_2-u_1\right)^2,
\end{align*}
where $v$ is some intermediate point between $u_1$ and $u_2$.
\end{proof}

\begin{lemma}
Let the Regularity conditions be fulfilled. Then there exists a constant
$d>0$, such that
\begin{equation}
\label{c}
\Pb_{\vartheta_{u_*}}\left\{ Z_n(u,u_*)>e^{-d\left|u-u_*\right|^2}\right\}\leq
e^{-d\left|u-u_*\right|^2}
\end{equation}
for all $u_*,u\in\mathbb{U}_n^+$ and sufficiently large value of $n$.
\end{lemma}

\begin{proof}
Using the Markov inequality, we get
$$
\Pb_{\vartheta_{u_*}}\left\{{Z}_n(u,u_*)>e^{-d\left|u-u_*\right|^2}\right\}\leq
e^{\frac12 d\left|u-u_*\right|^2}\Ex_{\vartheta_{u_*}}{Z}_n^{1/2}(u,u_*).
$$
According to~\cite[Lemma 1.1.5]{Kut98}, we have
\begin{align*}
\Ex_{\vartheta_{u_*}}&{Z}_n^{1/2}(u,u_*)\\
&=\exp\biggl\{-\frac{1}{2}\,n\int_0^{\tau}
  \Bigl(\frac{\lambda^{1/2}(\vartheta_1+u\varphi_n,t)}
       {\lambda^{1/2}(\vartheta_1+u_*\varphi_n,t)}
       -1\,\Bigr)^2\,\lambda(\vartheta_1+u_*\varphi_n,t)\,\mathrm{d} t\biggr\}\\
&=\exp\biggl\{-\frac{1}{2}\,n\int_0^{\tau}\Bigl(\lambda^{1/2}
       \bigl(\vartheta_1+u\varphi_n,t\bigr)
       -\lambda^{1/2}(\vartheta_1+u_*\varphi_n,t)\,\Bigr)^2\,\mathrm{d} t\biggr\}.
\end{align*}
Using the Taylor expansion we get
$$
\lambda^{1/2}\bigl(\vartheta_1+u\varphi_n,t\bigr)=
\lambda^{1/2}\bigl(\vartheta_1+u_*\varphi_n,t\bigr) +\frac{\varphi_n
  (u-u_*)}2\,\frac{\dot \lambda\left(\vartheta_v,t\right)}
       {\lambda^{1/2}\left(\vartheta_v,t\right)}\,,
$$
where $v$ is some intermediate point between $u_*$ and $u$. Hence, for
sufficiently large $n$ providing $\varphi_n\left|u-u_* \right|\leq \gamma $,
we have the inequality ${\rm I}\left(\vartheta_v\right)\geq \frac12{\rm
  I}\left(\vartheta _1\right)$, and we obtain
\begin{equation}
\label{a}
\begin{aligned}
\Ex_{\vartheta_{u_*}}{Z}_n^{1/2}(u,u_*)&\leq \exp\Biggl\{-\frac{1}{8{\rm
    I}\left(\vartheta _1\right)}\,\left|u-u_*\right|^2\,{\rm
  I}\left(\vartheta_v\right)\Biggr\}\\
&\leq\exp\Biggl\{-\frac{\left|u-u_*\right|^2}{16}\,\Biggr\}. 
\end{aligned}
\end{equation}
By Distinguishability condition, we can write
$$
g(\gamma)=\inf_{\varphi
  _n\left|u-u_*\right|>\gamma}\int_0^{\tau}\Bigl(\lambda^{1/2}\bigl(\vartheta_1+u\varphi_n,t\bigr)
-\lambda^{1/2}(\vartheta_1+u_*\varphi_n,t)\,\Bigr)^2\,\mathrm{d} t>0,
$$
and hence
$$
\int_0^{\tau}\Bigl(\lambda^{1/2}\bigl(\vartheta_1+u\varphi_n,t\bigr)
-\lambda^{1/2}(\vartheta_1+u_*\varphi_n,t)\,\Bigr)^2\,\mathrm{d} t\geq g(\gamma)\geq
g(\gamma)\frac{\varphi_n^2(u-u_*)^2}{\left(b-\vartheta_1\right)^2}
$$
and
\begin{equation}
\label{b}
\Ex_{\vartheta_{u_*}}{Z}_n^{1/2}(u,u_*)\leq
\exp\Biggl\{-\frac{g(\gamma)\left|u-u_*\right|^2}{2{\rm I}\left(\vartheta
  _1\right)\left(b-\vartheta_1\right)^2}\Biggr\}.
\end{equation}
So, putting
$$
d=\frac{2}{3}\min\left(\frac{1}{16}\,,\,\frac{g(\gamma)}{2{\rm
    I}\left(\vartheta _1\right)\left(b-\vartheta_1\right)^2}\right),
$$
the estimate \eqref{c} follows from \eqref{a} and \eqref{b}.
\end{proof}

The weak convergence of $Z_n\left(\cdot\right)$ now follows from \cite[Theorem
  1.10.1]{IH81}.

\section{Hypothesis testing}

In this section, we construct the Score Function test, the General Likelihood
Ratio test, the Wald test and two Bayes tests. For all these tests we describe
the choice of the thresholds and evaluate the limit power functions for local
alternatives.

\subsection{Score Function test}

Let us introduce the \textit{Score Function test} (SFT)
$$
\psi _n^\star\left(X^n\right)=\1_{\left\{\Delta _n\left(\vartheta
  _1,X^n\right)>z_\varepsilon \right\}},
$$
where $z_\varepsilon$ is the ($1-\varepsilon$)-quantile of the standard normal
distribution $\mathcal{ N}\left(0,1\right)$ and the statistic ${\Delta}
_n\left(\vartheta _1,X^n\right)$ is
$$
{\Delta} _n\left(\vartheta _1,X^n\right)=\frac{1}{\sqrt{n{\rm
      I}\left(\vartheta _1\right)}}\sum_{j=1}^{n}\int_{0}^{\tau
}\frac{\dot\lambda \left(\vartheta _1,t\right)}{\lambda \left(\vartheta
  _1,t\right)}\left[{\rm d}X_j\left(t\right)-\lambda \left(\vartheta
  _1,t\right){\rm d}t \right].
$$

The SFT has the following well-known properties (one can see, for example,
\cite[Theorem 13.3.3]{LR05} for the case of i.i.d.\ observations).

\begin{proposition}
The test $\psi _n^\star\left(X^n\right)\in \mathcal{ K}_\varepsilon$ and is
LAUMP. For its power function the following convergence hold:
$$
\beta\left(\psi _n^\star,u_*\right)\longrightarrow \beta
^\star\left(u_*\right)=\Pb\left(\Delta >z_\varepsilon -u_* \right),\quad
\Delta \sim \mathcal{ N}\left(0,1\right).
$$
\end{proposition}

\begin{proof}
The property $\psi _n^\star\left(X^n\right)\in \mathcal{ K}_\varepsilon$ follows
immediately from the asymptotic normality (under hypothesis $\mathscr{H}_1$)
$$
\Delta _n\left(\vartheta _1,X^n\right)\Longrightarrow \Delta .
$$
Further, we have (under alternative $\vartheta_{u_*}=\vartheta
_1+u_*\varphi_n$) the convergence
$$
\beta\left(\psi_n^\star,u_*\right)\longrightarrow \Pb\left( \Delta
+u_*>z_\varepsilon \right)= \beta^\star \left(u_*\right).
$$
This follows from the Le Cam's Third Lemma and can be shown directly as
follows. Suppose that the intensity of the observed Poisson process is
$\lambda \left(\vartheta _1+u_*\varphi _n,t\right)$, then we can write
\begin{align*}
{\Delta} _n\left(\vartheta _1,X^n\right)&=\frac{1}{\sqrt{n{\rm
      I}\left(\vartheta _1\right)}}\sum_{j=1}^{n}\int_{0}^{\tau
}\frac{\dot\lambda \left(\vartheta _1,t\right)}{\lambda
\left(\vartheta
  _1,t\right)}\left[{\rm d}X_j\left(t\right)-\lambda \left(\vartheta
  _1+u_*\varphi _n,t\right){\rm d}t \right]\\
&\quad +\frac{1}{\sqrt{n{\rm
      I}\left(\vartheta _1\right)}}\sum_{j=1}^{n}\int_{0}^{\tau
}\frac{\dot\lambda \left(\vartheta _1,t\right)}{\lambda
\left(\vartheta
  _1,t\right)}\left[\lambda \left(\vartheta _1+u_*\varphi _n,t\right)-\lambda
  \left(\vartheta _1,t\right)\right]{\rm d}t \\
&=\Delta _n^*\left(\vartheta_1,X^n\right)+\frac{u_*}{{n{\rm
I}\left(\vartheta
    _1\right)}}\sum_{j=1}^{n}\int_{0}^{\tau }\frac{\dot\lambda \left(\vartheta
  _1,t\right)^2}{\lambda \left(\vartheta _1,t\right)}{\rm d}t+o\left(1\right)\\
&=\Delta
_n^*\left(\vartheta_1,X^n\right)+u_*+o\left({1}\right)\Longrightarrow
\Delta +u_*.
\end{align*}

To show that the SFT is LAUMP, it is sufficient to verify that the limit of
its power function coincides (for each fixed value $u_*>0$) with the limit of
the power of the corresponding likelihood ratio (Neyman-Person) test (N-PT)
$\psi _n^*\left(X^n\right)$ . Remind that the N-PT is the most powerful for
each fixed (simple) alternative (see, for example, Theorem 13.3 in Lehman and
Romano~\cite{LR05}). Of course, the N-PT is not a real test (in our one-sided
problem), since for its construction one needs to know the value $u_*$ of the
parameter $u$ under alternative.

The N-PT is defined by
$$
\psi _n^*\left(X^n\right)=\1_{\left\{Z_n\left(u_*\right)>d_\varepsilon
  \right\}}+q_\varepsilon \1_{\left\{Z_n\left(u_*\right)=d_\varepsilon
  \right\}},
$$
where the threshold $d_\varepsilon $ and the probability $q_\varepsilon $ are
chosen from the condition $\psi _n^*\left(X^n\right)\in \mathcal{
  K}_\varepsilon $, that is,
$$
\Pb_{\vartheta _1}\left\{Z_n\left(u_*\right)>d_\varepsilon
\right\}+q_\varepsilon\Pb_{\vartheta
  _1}\left\{Z_n\left(u_*\right)=d_\varepsilon \right\} =\varepsilon.
$$
Of course, we can put $q_\varepsilon =0 $ because the limit random variable
$Z\left(u_*\right)$ has continuous distribution function.

The threshold $d_\varepsilon $ can be found as follows. The LAN of the family
of measures at the point $\vartheta _1$ allows us to write
\begin{align*}
\Pb_{\vartheta _1}\left(Z_n\left(u_*\right)>d_\varepsilon
  \right)&=\Pb_{\vartheta _1}\left(u_*\Delta _n\left(\vartheta
  _1,X^n\right)-\frac{u_*^2}{2}+r_n>\ln d_\varepsilon \right)\\
&\longrightarrow\Pb\left(u_*\Delta -\frac{u_*^2}{2}>\ln d_\varepsilon \right)=
  \Pb\left(\Delta >\frac{\ln d_\varepsilon}{u_*}
  +\frac{u_*}{2}\right)=\varepsilon .
\end{align*}
Hence, we have
$$
\frac{\ln d_\varepsilon}{u_*} +\frac{u_*}{2}=z
_\varepsilon\qquad\text{and}\qquad d_\varepsilon =\exp\left\{u_* z
_\varepsilon-\frac{u_*^2}{2}\right\}.
$$
Therefore the N-PT
$$
\psi_n^*\left(X^n\right)=\1_{\left\{Z_n\left(u_*\right)>\exp\bigl\{u_* z
  _\varepsilon-\frac{u_*^2}{2}\bigr\}\right\}}
$$
belongs to $\mathcal{ K}_\varepsilon $.

For the power of the N-PT we have (denoting as usually $\vartheta
_{u_*}=\vartheta _1+u_*\varphi _n$)
\begin{align*}
\beta\left(\psi _n^*,u_*\right)&=\Pb_{\vartheta
  _{u_*}}\left(Z_n\left(u_*\right)>d_\varepsilon \right)=\Pb_{\vartheta
  _{u_*}}\left( u_*\Delta _n\left(\vartheta _1,X^n\right)+r_n>u_* z
_\varepsilon \right)\\
&=\Pb_{\vartheta _{u_*}}\left( \Delta _n\left(\vartheta
_1,X^n\right)+\frac{r_n}{u_*}> z _\varepsilon \right)\!\longrightarrow\!
\Pb\left(\Delta+u_*>z _\varepsilon\right)=\beta ^\star\left(u_*\right).
\end{align*}
Therefore the limits of the powers of the tests $\psi _n^\star$ and $\psi
_n^*$ coincide, that is, the Score Function test is asymptotically as good as
the Neyman-Pearson optimal one. Note that the limits are valid for any
sequence of $0\leq {u_*}\leq K$.  So, for any $K>0$, we can choose a sequence
$\hat u_n\in \left[0,K\right]$ such that
$$
\sup_{0\leq {u_*}\leq K}\left|\beta\left(\psi _n^\star,{u_*}\right)-\beta
\left(\psi_n^*,{u_*}\right)\right|= \left|\beta\left(\psi _n^\star,\hat
u_n\right)-\beta\left(\psi _n^*,\hat u_n\right)\right|\longrightarrow 0,
$$
which represents the asymptotic coincidence of the two tests and concludes the
proof.
\end{proof}

\subsection{GLRT and Wald test}

Let us remind that the maximum likelihood estimator (MLE) $\hat\vartheta _n$
is defined by the equation:
$$
L\left(\hat\vartheta _n ,\vartheta _1,X^n\right)=\sup_{\vartheta \in \left[
    \vartheta _1,b\right)}L\left(\vartheta ,\vartheta _1,X^n\right),
$$
where the likelihood ratio function is
\begin{align*}
L\left(\vartheta ,\vartheta
_1,X^n\right)&=\exp\left\{\sum_{j=1}^{n}\int_{0}^{\tau}\ln\frac{\lambda
  \left(\vartheta ,t\right)}{\lambda \left(\vartheta_1 ,t\right)}\;{\rm
  d}X_j\left(t\right)\right.\\
&\qquad\qquad \qquad \left. -n\int_{0}^{\tau}\left[{\lambda \left(\vartheta
    ,t\right)}-{\lambda \left(\vartheta_1 ,t\right)}\right]{\rm
  d}t\right\},\qquad \vartheta \in \left[ \vartheta _1,b\right).
\end{align*}

The GLRT is
$$
\hat\psi_n\left(X^n\right)=\1_{\left\{Q\left(X^n\right)> h_\varepsilon
  \right\}},
$$
where
$$
Q\left(X^n\right)= \sup_{\vartheta \in \left[ \vartheta
    _1,b\right)}{L\left(\vartheta ,\vartheta_1
  ,X^n\right)}=L\left(\hat\vartheta _n ,\vartheta_1 ,X^n\right)
\quad\text{and}\quad h_\varepsilon=\exp\{z_\varepsilon^2/2\}.
$$

The Wald's test is based on the MLE $\hat\vartheta _n $ and is defined as
follows:
$$
\psi_n^\circ\left(X^n\right)=\1_{\left\{\varphi_n^{-1}\left(\hat\vartheta
  _n-\vartheta _1\right)> z_\varepsilon \right\}}.
$$

The properties of these tests are given in the following Proposition.

\begin{proposition}
The tests $\hat\psi_n\left(X^n\right)$ and $\psi_n^\circ\left(X^n\right)$ belong
to $\mathcal{K}_\varepsilon$, their power functions $\beta (\hat\psi_n,u_*)$ and
$\beta \left(\psi_n^\circ,u_*\right)$ converge to $\beta ^\star\left(u_*\right)$, and
therefore they are LAUMP.
\end{proposition}

\begin{proof}
Let us put $\vartheta =\vartheta _1+u\varphi_n$ and denote $\hat u_n=
\varphi_n^{-1}\left(\hat \vartheta_n -\vartheta _1\right)$. We have
\begin{align*}
\Pb_{\vartheta _1}\left\{\sup_{\vartheta\in \left[\vartheta
    _1,b\right)}L\left(\vartheta ,\vartheta_1 ,X^n\right)>h_\varepsilon
  \right\}&=\Pb_{\vartheta _1}\left\{\sup_{u\in
    \mathbb{U}_n^+}L\left(\vartheta_1+u\varphi _n ,\vartheta_1
  ,X^n\right)>h_\varepsilon \right\}\\
&=\Pb_{\vartheta
  _1}\left\{\sup_{u\in\mathbb{U}_n^+}Z_n\left(u\right)>h_\varepsilon \right\}.
\end{align*}

According to Theorem \ref{T1} (with $u_*=0$), we have the weak convergence
(under $\vartheta_1$) of the measure of the stochastic processes
$(Z_n\left(u\right),\ u\geq 0)$ to those of the process
$(Z\left(u\right),\ u\geq 0)$. This provides us the convergence of the
distributions of all continuous in uniform metric functionals. Hence
\begin{align*}
Q\left(X^n\right)&=\sup_{u>0}Z_n\left(u\right) \Longrightarrow
\sup_{u>0}Z\left(u\right)\\
&=\sup_{u>0}\ \exp\left\{u\Delta -\frac{u^2}{2}\right\}
=\exp\left\{\frac{\Delta^2}{2}\1_{\left\{\Delta\geq 0\right\}}\right\},
\end{align*}
which yields (we suppose that $\varepsilon \leq \frac{1}{2}$)
$$
\Ex_{\vartheta _1}\hat\psi_n\left(X^n\right)\longrightarrow
\Pb\left\{\Delta\1_{\left\{\Delta\geq 0\right\}}
>z_\varepsilon \right\} =\Pb\left\{\Delta
>z_\varepsilon \right\} =\varepsilon .
$$

Using the same weak convergence we obtain the asymptotic normality of the MLE
(see \cite{IH81} or \cite{Kut98}):
$$
\hat u_n =\frac{\hat \vartheta_n -\vartheta _1}{\varphi _n}\Longrightarrow
\hat u=\Delta\1_{\left\{\Delta\geq 0\right\}},
$$
and hence $\Ex_{\vartheta _1}\psi_n^\circ\longrightarrow\varepsilon$.  So both
$\hat\psi_n$ and $\psi_n^\circ$ belong to $\mathcal{K}_\varepsilon$.

Now, let us fix some $u_*>0$ and study the limit behavior of the power
functions of the tests.

Using the weak convergence of the likelihood ratio process under the
alternative $\vartheta_{u_*} =\vartheta _1+u_*\varphi_n$, we have
\begin{align*}
Q\left(X^n\right)=\sup_{u> 0}Z_n\left(u\right)\Longrightarrow\sup_{u>
  0}Z\left(u,u_*\right)&=\sup_{u>0}\ \exp\left\{u\Delta+u
u_*-\frac{u^2}{2}\right\}\\
&=\exp\left\{\frac{(\Delta+u_*)^2}{2}\1_{\left\{\Delta+u_*\geq
  0\right\}}\right\}.
\end{align*}

Hence (we suppose again that $\varepsilon \leq \frac{1}{2}$),
\begin{align*}
\beta\left(\hat\psi_n,u_*\right) &= \Pb_{\vartheta
  _{u_*}}\left\{Q\left(X^n\right)> h_\varepsilon \right\}\\
&\longrightarrow \Pb\left\{\left(\Delta+u_*\right)\1_{\left\{\Delta+u_*\geq
  0\right\}}> z_\varepsilon \right\}=\Pb\left\{\Delta > z_\varepsilon
-u_*\right\}\!=\!\beta ^\star\left(u_*\right).
\end{align*}

Similarly we have
$$
\beta\left(\psi_n^\circ,u_*\right)\longrightarrow
\Pb\left\{\left(\Delta+u_*\right)\1_{\left\{\Delta+u_*\geq
  0\right\}}>z_\varepsilon\right\}=\beta ^\star\left(u_*\right).
$$

Therefore the tests are LAUMP.
\end{proof}

\bigskip

\textbf{Example 1.} As the family of measures is LAN and the problem is
asymptotically equivalent to the corresponding hypothesis testing problem for
a Gaussian model, we propose here a similar test for Gaussian observations.

Suppose that the random variable $X\sim \mathcal{ N}\left(u,1\right)$ and we
have to test the hypothesis $\mathcal{ H}_1\ :\ u=0$ against $\mathcal{
  H}_2\ :\ u>0$. Then the SFT $\hat
\psi\left(X\right)=\1_{\left\{X>z_{\varepsilon }\right\}}$ is the uniformly
most powerful in the class of tests of size $\varepsilon $. Its power function
is $\beta \left(\hat\psi,u_*\right)=\Pb\left(\Delta >z_\varepsilon
-u_*\right)$. The log-likelihood function is
$$
L\left(u,X\right)=-\frac{1}{2}\ln\left(2\pi
\right)-\frac{1}{2}\left(X-u\right)^2
$$
The one-sided MLE $\hat u$ is given by
$$
\hat u=\argmax_{u\geq 0}L(X,u)=\max\{X,0\},
$$
and it is easy to see that the test $\hat \psi\left(X\right) $ and the Wald
test $\psi^\circ(X)=\1_{\left\{\hat u>z_\varepsilon \right\}}$ have identical
power functions.

\bigskip

Let us note, that the asymptotic equivalence to the SFT and the optimality is
a well known property of these tests in regular statistical experiments (see,
for example, \cite{LR05} and \cite{Kut13}). We present these properties here
in order to compare the asymptotics of these tests in regular and singular
situations (see \cite{DKY-2}). In particular, we will see that the asymptotic
properties of these tests in singular situations will be essentially
different.

\subsection{Bayes  tests}

Suppose now that the unknown parameter $\vartheta $ is a random variable with
\textit{a priori} density $p\left(\theta \right)$,
$\vartheta_1\leq\theta<b$. Here $p\left(\cdot \right)$ is a known continuous
function satisfying $p\left(\vartheta _1\right)>0$.  We consider two
approaches. The first one is based on the Bayes estimator and the second one
on the averaged likelihood ratio function.

\bigskip

The first test, wich we call BT1, is a Wald type test but based on the Bayes
estimator (BE) $\tilde\vartheta _n$:
$$
\tilde\psi _n\left(X^n\right)=\1_{\left\{\varphi _n^{-1}\left(\tilde\vartheta
  _n-\vartheta _1\right)>k_\varepsilon \right\}} .
$$

Remind that the BE for quadratic loss function is
$$
\tilde\vartheta _n=\int_{\vartheta _1}^{b}\theta\, p\left(\theta
|X^n\right){\rm d}\theta =\frac{\int_{\vartheta _1}^{b}\theta\, p\left(\theta
  \right)L\left(\theta ,\vartheta _1,X^n\right){\rm d}\theta }{\int_{\vartheta
    _1}^{b} p\left(\theta \right)L\left(\theta ,\vartheta _1,X^n\right){\rm
    d}\theta} .
$$
After the change of variables $\theta =\vartheta _1+v\varphi _n$ in the
integrals, we obtain the relation
$$
\varphi_n^{-1}\left(\tilde\vartheta_n-\vartheta_1\right)=
\frac{\int_{\mathbb{U}_n^+}v p\left(\vartheta _1+v\varphi
  _n\right)Z_n\left(v\right){\rm d}v}{\int_{\mathbb{U}_n^+}p\left(\vartheta
  _1+v\varphi _n\right) Z_n\left(v\right){\rm d}v}.
$$
The properties of $Z_n\left(\cdot \right)$ established in the proof of Theorem
\ref{T1} yield the following convergence in distribution under the hypothesis
$\mathscr{H}_1$ (see~\cite{IH81} or \cite{Kut98})
\begin{align*}
\varphi_n^{-1}\left(\tilde\vartheta_n-\vartheta_1\right)\Longrightarrow \tilde
u&=\frac{\int_{0}^{\infty }v Z\left(v\right){\rm d}v}{\int_{0}^{\infty }
  Z\left(v\right){\rm d}v}\\
&=\frac{\int_{0}^{\infty }
  \left(v-\Delta\right)\exp\left\{-\frac{\left(v-\Delta\right)^2}2\right\}{\rm
    d}v}{\int_{0}^{\infty }
  \exp\left\{-\frac{\left(v-\Delta\right)^2}2\right\}{\rm d}v}+\Delta\\
&=\frac{-\frac1{\sqrt{2\pi}}\exp
  \left\{-\frac{\left(v-\Delta\right)^2}2\right\}\Bigm|_{v=0}^{+\infty}}
{\frac1{\sqrt{2\pi}}\int_{0}^{\infty
  }\exp\left\{-\frac{\left(v-\Delta\right)^2}2\right\}{\rm d}v}+\Delta\\
&=\frac{\frac1{\sqrt{2\pi}}\exp\left\{-\frac{\Delta^2}2\right\}}
{\left(1-F\left(-\Delta\right)\right)}+\Delta
=\frac{f\left(\Delta\right)}{F\left(\Delta\right)}+\Delta,
\end{align*}
where $f\left(\cdot \right)$ and $F\left(\cdot \right)$ are the density and
the distribution function of the standard normal Gaussian random variable
$\Delta$. Hence, if we take $k_\varepsilon $ to be solution of the equation
$$
\Pb\left\{\frac{f\left(\Delta \right)}{F\left(\Delta \right)}+\Delta
>k_\varepsilon \right\}=\varepsilon,
$$
then the BT1 $ \tilde \psi _n$ belongs to $\mathcal{ K}_\varepsilon$.

A similar calculation under the alternative
$\vartheta_{u_*}=\vartheta+u_*\varphi_n$ allows us to evaluate the limit power
function of the BT1 as follows:
\begin{align*}
\beta\left(\tilde\psi
_n,u_*\right)&=\Pb_{\vartheta_{u_*}}\left\{\varphi_n^{-1}
\left(\tilde\vartheta_n-\vartheta_1\right)>k_\varepsilon\right\}\\
&\longrightarrow\Pb\left\{\frac{\int_{0}^{\infty } v Z\left(v,u_*\right){\rm
    d}v}{\int_{0}^{\infty } Z\left(v,u_*\right){\rm
    d}v}>k_\varepsilon\right\}\\
& =\Pb\left\{\frac{f\left(\Delta+u_*\right)}
          {F\left(\Delta+u_*\right)}+\Delta+u_*>k_\varepsilon\right\}.
\end{align*}

\bigskip

Another possibility in Bayesian approach is to define the test as the test
with the minimal mean error of the second kind. For a test $\bar\psi _n$, let
us denote $\alpha \left(\bar\psi _n,\theta \right)=1-\beta \left(\bar\psi
_n,\theta \right) $ the error of the second kind and introduce the mean error
of the second kind:
$$
\alpha \left(\bar\psi _n\right)=\int_{\vartheta _1}^{b} \alpha \left(\bar\psi
_n,\theta \right)\,p\left(\theta \right){\rm d}\theta .
$$
The Bayes test $\tilde\psi_n^\star\left(X^n\right)$ is defined as the test
which minimizes this mean error:
$$
\alpha \left(\tilde\psi_n^\star\right)=\inf_{ \bar\psi _n\in \mathcal{
    K}_\varepsilon }\alpha \left(\bar\psi _n\right) .
$$

We can rewrite the above integral as follows
\begin{align*}
\int_{\vartheta _1}^{b}\left(1- \Ex_\theta \bar\psi
_n\left(X^n\right)\right)\,p\left(\theta \right)\,{\rm
  d}\theta&=\int_{\vartheta _1}^{b} \int_{}^{}\left(1-\bar\psi
_n\left(x^n\right)\right)\mathrm{d}{\Pb_\theta}\; p\left(\theta \right)\mathrm{d}\theta\\
&= \int_{}^{}\left(1-\bar\psi_n\left(x^n\right)\right)\mathrm{d}{\PP}
=\EE\left(1-\bar\psi_n\left(X^n\right)\right),
\end{align*}
where we denoted $\Pb_\theta$ the distribution of the sample $X^n$ and
$$
\PP\left(X^n\in A\right)=\int_{\vartheta _1}^{b}\Pb_{\theta }\left(X^n\in
A\right)\,p\left(\theta \right)\mathrm{d}\theta .
$$

The averaged power $\beta\left(\bar\psi
_n\right)=\EE\bar\psi_n\left(X^n\right)$ is the same as if we have two simple
hypothesis. Under $\mathscr{ H}_1$ we observe a Poisson process of intensity
function $\lambda \left(\vartheta _1,\cdot\right)$, and under the alternative
$\mathscr{ H}_2 $ the observed point process has random intensity and its
measure is~$\PP$. This process is a mixture $\bigl($according to the density
$p(\theta)\bigr)$ of inhomogeneous Poisson processes with intensities $\lambda
\left(\theta ,\cdot \right)$, $\theta\in \left(\vartheta _1,b\right)$. This
means that we have two simple hypotheses and the most powerful
(Neyman-Pearson) test is of the form
$$
\tilde\psi_n^*=\1_{\left\{ \tilde L\left(X^n\right)>\tilde m_\varepsilon
  \right\}},\qquad \Ex_{\vartheta _1}\tilde\psi_n^*\left(X^n\right)=
\varepsilon ,
$$
where the averaged likelihood ratio
$$
\tilde L\left(X^n\right)= \frac{\mathrm{d}{\PP}\hfill}{\mathrm{d}{\Pb_{\vartheta _1}}}
\left(X^n\right) =\int_{\vartheta
  _1}^{b}\frac{\mathrm{d}{\Pb_{\theta}}\hfill}{\mathrm{d}{\Pb_{\vartheta _1}}}
\left(X^n\right)\,p\left(\theta \right)\,{\rm d}\theta.
$$
To study this test under hypothesis we change the variables:
$$
\tilde L\left(X^n\right)=\int_{\vartheta _1}^{b}L\left(\theta ,\vartheta
_1,X^n\right)\,p\left(\theta \right)\,{\rm d}\theta
=\varphi_n\int_{0}^{\varphi_n^{-1}\left(b-\vartheta _1\right)}
Z_n\left(v\right)\,p\left(\vartheta _1+v\varphi_n\right)\,{\rm d}v.
$$
The limit of the last integral was already described above and this allow us
to write
\begin{align*}
&{R}_n\left(X^n\right)= \frac{\tilde L\left(X^n \right)}{p\left(\vartheta
    _1\right)\varphi _n}=\frac1{p\left(\vartheta
    _1\right)}\int_{0}^{\varphi_n^{-1}\left(b-\vartheta _1\right)} e^{v\Delta
    _n-\frac{v^2}{2}+r_n }\,p\left(\vartheta _1+v\varphi_n\right)\,{\rm d}v\\
&\quad\Longrightarrow \int_{0}^{\infty } e^{v\Delta -\frac{v^2}{2} }\,{\rm
    d}v=e^{\frac{\Delta ^2}{2 } } \int_{ -\Delta }^{\infty }
  {e^{-\frac{y^2}{2} }\,{\rm d}y}{}=\sqrt{2\pi }\,e^{\frac{\Delta
      ^2}{2}}\left(1-F\left(-\Delta \right)\right)=\frac{F\left(\Delta
    \right)}{f\left(\Delta \right)},
\end{align*}
where $F\left(\cdot \right)$ and $f\left(\cdot \right)$ are again the
distribution function and the density of the standard Gaussian random variable
$\Delta$. Hence, if we take $m_\varepsilon $ to be solution of the equation
$$
\Pb\left\{\frac{F\left(\Delta \right)}{f\left(\Delta \right)}>m_\varepsilon
\right\}=\varepsilon,
$$
then the test $\tilde\psi_n^\star\left(X^n\right)=\1_{\left\{R_n>m_\varepsilon
  \right\}} $, which we call BT2, belongs to $\mathcal{ K}_\varepsilon$ and
coincides with the test $\tilde\psi_n^*\left(X^n\right)$ if we put $\tilde
m_\varepsilon=m_\varepsilon\, p\left(\vartheta _1\right)\varphi _n$.

A similar calculation under the alternative
$\vartheta_{u_*}=\vartheta+u_*\varphi_n$ allows us to evaluate the limit power
function of the BT2 as follows:
$$
\beta\left(\tilde\psi_n^\star,u_*\right)=\Pb_{\vartheta _{u_*}}\left\{
R_n>m_\varepsilon \right\}\longrightarrow \Pb\left\{ \frac{F\left(\Delta+u_*
  \right)}{f\left(\Delta+u_* \right)}>m_\varepsilon \right\}.
$$

\section{Simulations}

Below we present the results of numerical simulations for the power functions
of the tests. We observe $n$ independent realizations $X_j=\left(X_j(t),\ t
\in \left[0,3\right]\right)$, $j=1,...,n$, of inhomogeneous Poisson process of
intensity function
$$
\lambda \left(\vartheta ,t\right)=3\cos^2(\vartheta t)+1,\qquad 0\leq t \leq
3,\qquad \vartheta \in\left[3,7\right),
$$
where $\vartheta_1=3$. The Fisher information at the point $\vartheta_1$ is
${\rm I}\left(\vartheta_1\right) \approx 19.82$.  Recall that all our tests
(except Bayes tests) in regular case are LAUMP. Therefore they have the same
limit power function. Our goal is to study the power functions of different
tests for finite $n$.

The normalized likelihood ratio $Z_n(u)$ is given by the expression
\begin{align*}
Z_n(u)&=\exp\biggl\{\varphi_n\sum_{j=1}^n\int_0^3
\ln\frac{3\cos^2\left(\left(3+u\varphi_n\right)t\right)+1}
        {3\cos^2\left(3t\right)+1}\,\mathrm{d} X_j\left(t \right)\\
&\qquad\qquad-\frac{3n}{4\left(3+u\varphi_n\right)}
\sin\left(6\left(3+u\varphi_n\right)\right)+\frac{n}4\sin(18)\biggr\},
\end{align*}
where $\varphi _n=\left(19.82\,n\right)^{-1/2}$.

The numerical simulation of the observations allows us to obtain the power
functions presented in Figures~\ref{PF_Regular_2}
and~\ref{PF_Regular_GLRT_Wald}.  For example, the computation of the numerical
values of the power function of the SFT was done as follows. We define an
increasing sequence of $u$ beginning at $u=0$. Then, for every~$u$, we
simulate $N$ i.i.d.\ observations of n-tuples of inhomogeneous Poisson
processes $X^{n,i}$, $i=1,...,N$, with the intensity function
$\lambda\left(3+u\varphi_n,t\right)$ and calculate the corresponding
statistics $\Delta_n(3,X_{n,i})$, $i=1,...,N$. The empirical frequency of
acceptation of the alternative gives us an estimate of the power function:
$$
\beta\left(\psi _n^\star,u\right)\approx \frac{1}{N}\sum_{i=1}^N
\1_{\left\{\Delta_n(3,X_{n,i})>z_\varepsilon\right\}}.
$$
We repeat this procedure for different values of $u$ until the values of
$\beta\left(\psi _n^\star,u\right)$ become close to $1$.

\begin{figure}[htb]
\begin{center}
\includegraphics[width=\textwidth]{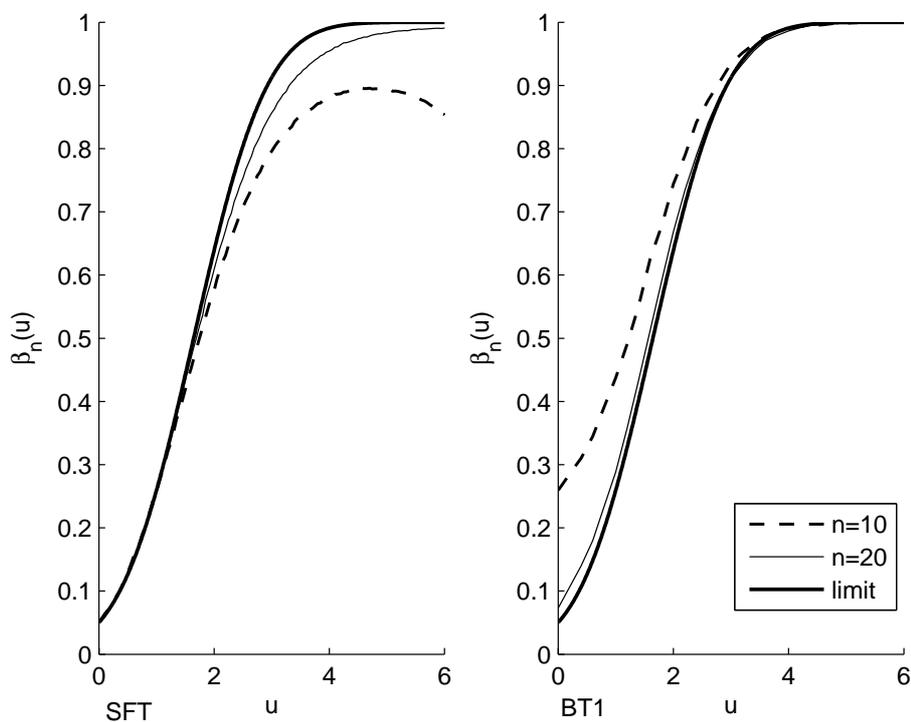}
\caption{\label{PF_Regular_2}Power functions of SFT and BT1}
\end{center}
\end{figure}

\begin{figure}[htb]
\begin{center}
\includegraphics[width=\textwidth]{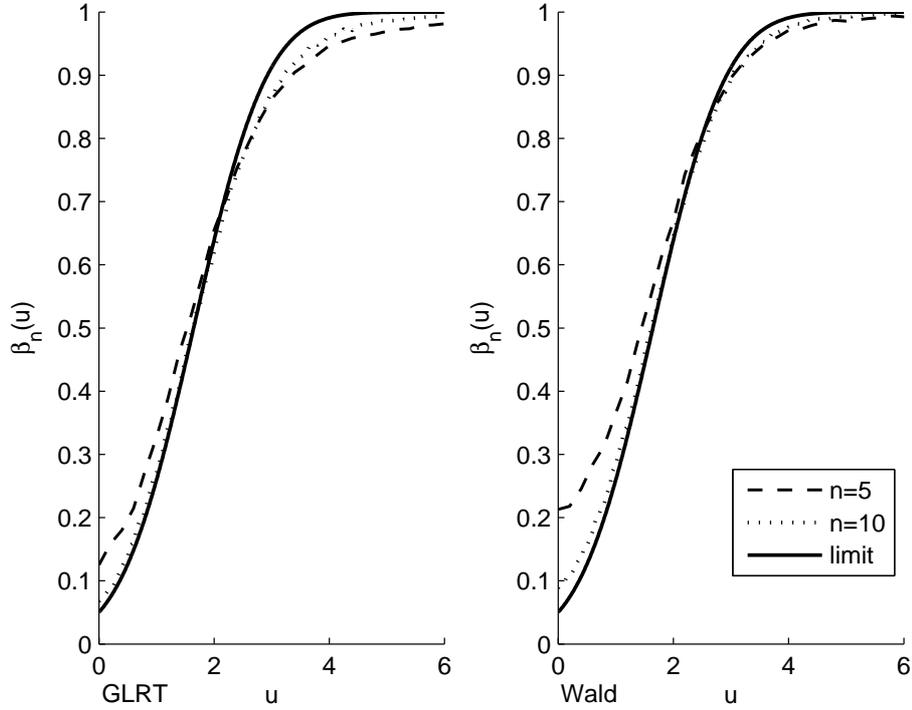}
\caption{\label{PF_Regular_GLRT_Wald}Power functions of GLRT and WT}
\end{center}
\end{figure}

In the computation of the power function of the Bayes test BT1, we take as
{\it a priori} law the uniform distribution, that is, $\vartheta\sim
\mathcal{U}([3,7])$. The thresholds of the BT1 are obtained by simulating
$M=10^5$ random variables $\Delta_i \sim \mathcal{ N}\left(0,1 \right)$,
$i=1,\ldots,M$, calculating for each of them the quantity
$$
\frac{f\left(\Delta_i\right)}{F\left(\Delta_i\right)}+\Delta_i,\qquad
i=1,\ldots,M,
$$
and taking the $\left(1-\varepsilon\right)M$-th greatest between them. Some of
the thresholds are presented in Table~\ref{Thr_BT1}.

\begin{table}[htb]
\begin{center}
\begin{tabular}{|c|c|c|c|c|c|c|c|}
  \hline $\varepsilon$ & 0.01 & 0.05 & 0.10 & 0.2 & 0.4 & 0.5 \\ \hline
  $k_\varepsilon$ & 2.325 & 1.751 & 1.478 & 1.193 & 0.895& 0.794 \\ \hline
\end{tabular}
\end{center}
\caption{\label{Thr_BT1}Thresholds of BT1}
\end{table}

Note that for the small values of $n$, under alternative, the power function
of SFT starts to decrease (see Figure~\ref{PF_Regular_GLRT_Wald}). This
interesting fact can be explained by the strongly non linear dependence of the
likelihood ratio on the parameter. The test statistic
$\Delta_n=\Delta_n\left(3,X^n\right)$ can be rewritten as follows:
\begin{align*}
\Delta_n&=\varphi_n\sum_{j=1}^n\int_{0}^{T }\frac{\dot\lambda \left(\vartheta
  _1,t\right)}{\lambda \left(\vartheta_1,t\right)}\left[{\rm d}X_j(t)-\lambda
  \left(\vartheta _1+u\varphi_n,t\right){\rm d}t \right]\\
 &\qquad+\sqrt{\frac{n}{{\rm
      I}\left(\vartheta_1\right)}}\int_{0}^{T}\frac{\dot\lambda
  \left(\vartheta _1,t\right)}{\lambda
  \left(\vartheta_1,t\right)}\left[\lambda \left(\vartheta
  _1+u\varphi_n,t\right)-\lambda \left(\vartheta_1,t\right) \right]{\rm d}t\\
&=-3\varphi_n\sum_{j=1}^n\int_{0}^{3}\frac{t\sin(6
  t)}{3\cos^2(3\,t)+1}\left[{\rm
    d}X_j\left(t\right)-\left(3\cos^2\!\left(\left(3+u\varphi_n\right)t\right)\!+1\right){\rm
    d}t \right]\\
 &\qquad+9\sqrt{\frac{n}{{\rm I}\left(\vartheta_1\right)}}\!\int_{0}^{3
}\!\frac{t\sin(6 t)}{3\cos^2(3\,t)+1}\times\left[\cos^2(3\,t)-
  \cos^2\left(\left(3+u\varphi_n\,\right)t\right)\right]{\rm d}t.
\end{align*}
The last integral becomes negative for some values of $u$, which explains the
loss of power of the SFT (for $n=10$).

\section{Acknowledgements}

This study was partially supported by Russian Science Foundation (research
project No. 14-49-00079). The authors thank the Referee for helpful comments.

\end{document}